\documentclass{amsart}

\usepackage{amsfonts}
\usepackage{latexsym,amssymb}
\usepackage{amsmath}
\usepackage[utf8]{inputenc}
\usepackage{color}

\newtheorem{theorem}{Theorem}[section]
\newtheorem{lemma}[theorem]{Lemma}
\newtheorem{proposition}[theorem]{Proposition}
\newtheorem{corollary}[theorem]{Corollary}
\newtheorem{remark}[theorem]{Remark}

\catcode`@=11 \@addtoreset{equation}{section} \catcode`@=12

\begin{document}

\title[Kirchhoff-Schr\"odinger equation with exponential critical growth]{Kirchhoff-Schr\"odinger equations in $\mathbb{R}^2$ with critical exponential growth and indefinite potential}

\author[M.F. Furtado]{Marcelo F. Furtado}
 \address{Universidade de Brasília, Departamento de Matem\'atica}
 \email{mfurtado@unb.br}

\author[H.R. Zanata]{Henrique R. Zanata}
 \address{Universidade Federal Rural do Semi-árido, Campus Caraúbas-RN, Brazil}
 \email{zanata@ufersa.edu.br}


\begin{abstract}
We obtain the existence of ground state solution for the nonlocal problem
$$
m\left(\int_{\mathbb{R}^2}(|\nabla u|^2 + b(x)u^2) \textrm{d}x\right)(-\Delta u + b(x)u) = A(x)f(u) \ \ \ \textrm{in} \ \ \ \mathbb{R}^2, 
$$
where $m$ is a Kirchoff-type function, $b$ may be negative and noncoercive, $A$ is locally bounded and the function $f$ has critical exponential growth. We also obtain new results for the  classical Schr\"odinger equation, namely the local case $m\equiv 1$. In the proofs we apply Variational Methods beside a new Trudinger-Moser type inequality.
\end{abstract}

\maketitle

\section{Introduction} 

We study the problem
$$
m\left(\int_{\mathbb{R}^2}(|\nabla u|^2 + b(x)u^2) \textrm{d}x\right)(-\Delta u + b(x)u) = A(x)f(u) \ \ \ \textrm{in} \ \ \ \mathbb{R}^2, \leqno{(P)}
$$
where $m : [0,\infty) \rightarrow (0,\infty)$ and $f : \mathbb{R} \rightarrow [0,\infty)$ are continuous functions and $b,A \in L_{\textrm{loc}}^{\infty}(\mathbb{R}^2)$. The potential $b$ may vanish on sets of positive measure or even be negative and the nonlinearity $f$ has critical growth. We look for solutions in the subspace of $W^{1,2}(\mathbb{R}^2)$ given by $$H := \left\{u \in W^{1,2}(\mathbb{R}^2) : \int_{\mathbb{R}^2} b(x)u^2 \textrm{d}x < \infty\right\}.$$

Due to the presence of the term $m(\int_{\mathbb{R}^2}(|\nabla u|^2 + b(x)u^2) \textrm{d}x)$ the equation is $(P)$ no longer a pointwise identity and therefore the problem is called nonlocal. In \cite{Kir}, G. Kirchhoff presented his study on transverse vibrations of elastic strings and proposed a hyperbolic equation of the type 
\begin{equation}\label{eq. orig. Kirchhoff}
\dfrac{\partial^2u}{\partial t^2} - \left(k_1 + k_2\int_0^L\left|\dfrac{\partial u}{\partial x}\right|^2 \textrm{d}x\right)\dfrac{\partial^2u}{\partial x^2} = 0,
\end{equation}
where $k_1$, $k_2$ e $L$ are positive constants. This extend the classical D'Alembert wave equation by considering the effects of
the changes in the length of the strings during the vibrations. So,  more general versions of \eqref{eq. orig. Kirchhoff} and the corresponding stationary equations have been called Kirchhoff equations and became subject of intense research mainly after the works of S.I. Pohozaev \cite{Poh} and J.-L. Lions \cite{JLio}. Variational Methods have been used by many authors to obtain results of existence and multiplicity of solutions for stationary Kirchhoff equations since the pioneering work of C.O. Alves et al. \cite{AlvCorMa}.

In order to present the conditions on the nonlocal term $m$ we first define $M(t):=\int_0^t m(\tau) \textrm{d}\tau$, $t \geq 0$. The hypotheses on $m : [0,\infty) \rightarrow (0,\infty)$ are:
\begin{itemize} 
\item[$(m_1)$] $\displaystyle m_0 := \inf_{t \geq 0} m(t) > 0$;
\item[$(m_2)$] for any $t_1,t_2 \geq 0$, it holds $$M(t_1+t_2) \geq M(t_1) + M(t_2);$$
\item[$(m_3)$] $\dfrac{m(t)}{t}$ is decreasing in $(0,\infty)$.
\end{itemize}
Condition $(m_2)$ is valid, for example, if $m$ is nondecreasing. The typical example of function satisfying $(m_1)-(m_3)$ is $m(t) = \alpha + \beta t$, with $\alpha>0$ and $\beta \geq 0$. Other examples are $m(t) = \alpha + \beta t^{\delta}$, with $\delta \in (0,1)$, $m(t) = \alpha(1+\log(1+t))$ or $m(t) = \alpha + \beta e^{-t}$. 

Concerning the potential  $b \in L_{\textrm{loc}}^{\infty}(\mathbb{R}^2)$ we set
$$\lambda_1^b := \inf\left\{\int_{\mathbb{R}^2}(|\nabla u|^2 + b(x)u^2) \textrm{d}x : u \in H \mbox{ and } \left\|u\right\|_{L^2(\mathbb{R}^2)} = 1\right\}$$
and, for each $\Omega \subset \mathbb{R}^2$ open and nonempty,
$$\nu_b(\Omega) := \inf\left\{\int_{\mathbb{R}^2}(|\nabla u|^2 + b(x)u^2) \textrm{d}x : u \in W_0^{1,2}(\Omega) \mbox{ and } \left\|u\right\|_{L^2(\Omega)} = 1\right\}
$$ 
and $\nu_b(\varnothing) = \infty$. The hypotheses on $b$ are:
\begin{itemize} 
\item[$(b_1)$] $\lambda_1^b > 0$;
\item[$(b_2)$] $\displaystyle \lim_{r\rightarrow\infty}\nu_b\left(\mathbb{R}^2 \backslash \overline{\{x \in \mathbb{R}^2 : |x|<r\}}\right) = \infty$;
\item[$(b_3)$] there exists $B_0 > 0$ such that $$b(x) \geq -B_0, \quad \forall\,  x \in \mathbb{R}^2.$$  
\end{itemize}

For the function $A \in L_{\textrm{loc}}^{\infty}(\mathbb{R}^2)$ we suppose that
\begin{itemize} 
\item[$(A_1)$] $A(x) \geq 1$ for any $x \in \mathbb{R}^2$;  
\item[$(A_2)$] there exists $\beta_0 > 1$, $C_0 > 0$ and $R_0 > 0$ such that $$A(x) \leq C_0 \left[ 1 + (b^+(x))^{1/\beta_0}\right], \quad \forall \, x \in \mathbb{R}^2 \setminus B_{R_0}(0),$$ where $b^+(x) := \max\{0,b(x)\}$. 
\end{itemize}

Conditions $(b_1)-(b_3)$ and $(A_1)-(A_2)$ were first considered by B. Sirakov \cite{Sir} in the study of a class of subcritical Schr\"odinger equations in dimension $N \geq 3$. These hypotheses ensure that $H$ is a Hilbert space with inner product given by
$$\left\langle u,v \right\rangle_H = \int_{\mathbb{R}^2}(\nabla u \cdot \nabla v + b(x)uv) \textrm{d}x,\quad \forall \, u,\, v \in H,$$
and norm $\left\|u\right\|_H = \sqrt{\left\langle u,u \right\rangle_H}$. Moreover $H$ is continuously embedded into $W^{1,2}(\mathbb{R}^2)$ and, for every $p \geq 2$, compactly embedded into the weighted Lebesgue space 
$$L_A^p(\mathbb{R}^2) := \left\{u:\mathbb{R}^2\rightarrow\mathbb{R} \ \textrm{measurable} \ : \ \int_{\mathbb{R}^2} A(x)|u|^p \textrm{d}x < \infty\right\},$$
which is a Banach space when endowed with the norm
$$\left\|u\right\|_{L_A^p(\mathbb{R}^2)} = \left(\int_{\mathbb{R}^2} A(x)|u|^p \textrm{d}x\right)^{1/p}.$$ 
For the proof of these embeddings, see \cite[Sections 2 and 3]{Sir}. By $(A_1)$, $L_A^p(\mathbb{R}^2) \hookrightarrow L^p(\mathbb{R}^2)$ and, consequently, the embedding $H \hookrightarrow L^p(\mathbb{R}^2)$ is also compact. In order to guarantee the compacity of this last embedding, one normally use the conditions $b(x) \geq b_0 > 0$ and
\begin{equation}\label{b coercivo} 
\lim_{|x|\rightarrow \infty}b(x) = \infty, \ \textrm{or} \ 1/b \in L^1(\mathbb{R}^2), \ \textrm{or} \ \mbox{meas}(\Omega_{b,K})<\infty \ \forall \, K>0,
\end{equation}
where $\Omega_{b,K} := \{x \in \mathbb{R}^2 : b(x)<K\}$. A weaker geometric condition that implies on $(b_2)$ is (see \cite[Theorem 1.4]{Sir}): for any $K>0$, any $r>0$ and any sequence $(x_n)\subset\mathbb{R}^2$ with $\displaystyle \lim_{n\rightarrow\infty}|x_n|=\infty$, we have
\begin{equation}\label{cond. geom. sobre b}
\lim_{n\rightarrow\infty} \mbox{meas}(\Omega_{b,K} \cap B_r(x_n)) = 0.
\end{equation}
A potential satisfying the above condition is $b(x)=b(x_1,x_2)=|x_1x_2|$. Since $(b_2)$ and $(b_3)$ are sufficient conditions for $\lambda_1^b$ to be achieved (see \cite[Proposition 2.2]{Sir}),  it is easy to see that this potential also satisfies $(b_1)$. Moreover, since for any constant $C \in \mathbb{R}$ we have $\Omega_{b-C,K} = \Omega_{b,K+C}$ and $\lambda_1^{b-C} = \lambda_1^b - C$, other potential satisfying $(b_1)-(b_3)$ is $b(x)=|x_1x_2| - C$, for certain values of $C$. Notice these two examples do not satisfy \eqref{b coercivo}.

Embedding $H \hookrightarrow W^{1,2}(\mathbb{R}^2)$ implies that, for some constant $\zeta>0$,
\begin{equation}\label{imersao H em D^(1,2)}
\left\|u\right\|_H \geq \zeta\left\|\nabla u\right\|_{L^2(\mathbb{R}^2)}, \ \ \forall \ u \in H.
\end{equation}
If $b \leq 0$ on some set with positive measure, then we cannot have $\zeta>1$. However, we can consider $\zeta=1$ if
\begin{itemize}
\item[$(\widehat{b_3})$] $b(x) \geq 0$ for any $x \in \mathbb{R}^2$.
\end{itemize} 

Concerning the nonlinearity $f : \mathbb{R} \rightarrow [0,\infty)$, we first suppose that $f(s) = 0$, for any $s \leq 0$, and define $F(s) := \int_0^s f(\tau) \textrm{d}\tau$, $s \in \mathbb{R}$. The main hypotheses on $f$ are:  
\begin{itemize}  
\item[$(f_1)$] there exists $\alpha_0 > 0$ such that $$\lim_{s \rightarrow \infty} \dfrac{f(s)}{e^{\alpha s^2}} = \left\{ \begin{array}{ll} 0, & \textrm{if} \ \alpha > \alpha_0, \\ \infty ,& \textrm{if} \ \alpha < \alpha_0; \end{array}  \right.$$ 

\item[$(f_2)$] there exists $s_0, K_0 > 0$ such that 
$$
F(s) \leq K_0f(s), \quad \forall \ s \geq s_0;
$$

\item[$(f_3)$] there exists $\theta_0 > 4$ such that 
$$
\theta_0 F(s) \leq sf(s), \quad \forall \ s > 0;
$$

\item[$(f_4)$] $\dfrac{f(s)}{s^3}$ is positive and nondecreasing  in $(0,\infty)$.
\end{itemize}
If $\theta>4$, an example of function $f$ satisfying $(f_1)-(f_4)$ is
$$f(s) = \dfrac{d}{ds}\left(\dfrac{s^{\theta}}{\theta}(e^{s^2}-1)\right) = s^{\theta - 1}(e^{s^2}-1) + \dfrac{2s^{\theta+1}}{\theta}e^{s^2}.$$

According to $(f_1)$ we are dealing with a function with critical growth. This notion of criticality was originally motivated by the Trudinger-Moser inequality (see \cite{Mos,Tru}), which states that $W_0^{1,2}(\Omega)$ is continuously embedded into the Orlicz space $L_{\phi_\alpha}(\Omega)$ associated with the function $\phi_\alpha(t) := e^{\alpha t^2}-1$, $t \in \mathbb{R}$, for $0< \alpha \leq 4\pi$ and any bounded domain $\Omega \subset \mathbb{R}^2$. This result has been generalized in many ways (see \cite{Cao,DoO,Ruf,AdiSan,LiRuf,SouDoO1,SouDoO2} and references therein). 
Here, we prove a version of that result for functions belonging to the  space $H$ (see Lemma \ref{Lema 3}).

The main difficulty in dealing with critical growth is the lack of compactness from the embeddings
of the Sobolev spaces into Orlicz spaces $L_{\phi_\alpha}$. In \cite[subsection I.7]{PLio}, P.-L. Lions proved a concentration-compactness result that allow us to overcome this trouble in $W_0^{1,2}(\Omega)$, $\Omega \subset \mathbb{R}^2$ bounded domain. This result have had many generalizations and applications in recent years (see \cite{LamLu,Yang1,Yang2,CerCiaHen,DoOSouMedSev} and references therein). Corollary \ref{Corolario 1} in  next section is a version of the result of P.-L. Lions for the space $H$.

Before stating our results, we need to fix some notations:
\begin{equation} \label{def_sp}
S_p := \inf_{u \in H\backslash\{0\}} \dfrac{\left\|u\right\|_H}{\left\|u\right\|_{L^p(\mathbb{R}^2)}}, \ \ p \geq 2,
\end{equation}
$$
C_p := \inf\left\{C>0 : pM(t^2S_p^2) - 2Ct^p \leq pM\left(\dfrac{4\pi\zeta^2}{\alpha_0}\right), \ \forall \ t > 0\right\}, \ \ p>4.
$$
The values $S_p$ and $C_p$ are finite, for $p \geq 2$ and $p>4$ respectively, due to the embedding $H \hookrightarrow L^p(\mathbb{R}^2)$ and the hypothesis $(m_3)$, which implies that $m(t)<m(1)t$ for any $t>1$.

Our main results for the problem $(P)$ can be stated as follows:

\begin{theorem}\label{Teorema exist. sol. caso zeta<1}
Suppose that $(m_1)-(m_3)$, $(b_1)-(b_3)$, $(A_1)-(A_2)$ and $(f_1)-(f_4)$ are satisfied. Suppose also that
\begin{itemize}
\item[$(f_5)$] there exists $p_0>4$ such that $$f(s) > C_{p_0}s^{p_0-1}, \ \ \forall \ s > 0.$$
\end{itemize}
Then problem $(P)$ has a nonnegative \textit{ground state} solution. 
\end{theorem}

\begin{theorem}\label{Teorema exist. sol. caso zeta=1}
Suppose that $(m_1)-(m_3)$, $(b_1)-(b_2)$, $(\widehat{b_3})$, $(A_1)-(A_2)$ and $(f_1)-(f_4)$ are satisfied. Suppose also that 
\begin{itemize}
\item[$(f_6)$] there exists $\gamma_0 > 0$ such that
$$\liminf_{s \rightarrow \infty} \dfrac{sf(s)}{e^{\alpha_0 s^2}} \geq \gamma_0 > 4\alpha_0^{-1}m\left(\dfrac{4\pi}{\alpha_0}\right)\inf_{R>0}\left\{R^{-2}e^{R^2M_R/2}\right\},$$
where $M_R := \left\|b\right\|_{L^{\infty}(B_R(0))}$.
\end{itemize}
Then problem $(P)$ has a nonnegative \textit{ground state} solution. 
\end{theorem}

Hypotheses $(m_3)$ and $(f_4)$ ensure that the solutions given by Theorems \ref{Teorema exist. sol. caso zeta<1} and \ref{Teorema exist. sol. caso zeta=1} are ground state solutions. However, as we will see in the proofs, we still obtain nonnegative nontrivial solution for the problem $(P)$, not necessarily ground state, if we replace $(m_3)$ and $(f_4)$ by weaker conditions, namely: 
\begin{itemize}
\item[$(m_3^*)$] there exist constants $a_1 > 0$ and $T>0$ such that $$m(t) \leq a_1t, \ \ \forall \ t \geq T;$$
\item[$(f_4^*)$] $\displaystyle \lim_{s \rightarrow 0^+} \dfrac{f(s)}{s} = 0$
\end{itemize}
and the conditions of monotonicity  given in the conclusion of  Lemma \ref{Lema 9} in the next section. Specifically in the case of Theorem \ref{Teorema exist. sol. caso zeta=1}, this replacement allow us to consider functions $f$ that vanish on some neighborhood of origin.

The ideias used here permit us to obtain new results even in the local case. Actually, if $m \equiv 1$, equation in $(P)$ is reduced to the Schr\"odinger equation
$$
-\Delta u + b(x)u = A(x)f(u) \ \ \ \textrm{in} \ \ \ \mathbb{R}^2.
\leqno{(\widehat{P})}
$$
In this case, instead of $(f_3)$ and $(f_4)$, we consider the hypotheses
\begin{itemize}  
\item[$(\widehat{f_3})$] there exists $\widehat{\theta_0} > 2$ such that $$\widehat{\theta_0}F(s) \leq sf(s), \quad \forall\, \ s > 0;$$
\item[$(\widehat{f_4})$] $\dfrac{f(s)}{s}$ is positive and nondecreasing in $(0,\infty)$. 
\end{itemize}
In contrast to $(f_4)$, hypothesis $(\widehat{f_4})$ does not imply on $(f_4^*)$. Setting, for $q > 2$,
$$
\widehat{C}_q := \inf\left\{C>0 : qS_q^2t^2 - 2Ct^q \leq \dfrac{4\pi q\zeta^2}{\alpha_0}, \ \forall \ t > 0\right\} = S_q^q\left(\dfrac{\alpha_0(q-2)}{4\pi q\zeta^2}\right)^{(q-2)/2},
$$
the main results for problem $(\widehat{P})$ can be stated as follows:

\begin{theorem}\label{Teo. exist. sol. eq. Schr. caso zeta<1}
Suppose that $(b_1)-(b_3)$, $(A_1)-(A_2)$, $(f_1)-(f_2)$, $(\widehat{f_3})$ and $(f_4^*)$ are satisfied. Suppose also that 
\begin{itemize}
\item[$(\widehat{f_5})$] there exists $q_0>2$ such that $$f(s) > \widehat{C}_{q_0}s^{q_0-1}, \ \ \forall \ s > 0.$$
\end{itemize}
Then problem $(\widehat{P})$ has a nonnegative nontrivial weak solution. If, in addition, $f$ satisfies $(\widehat{f_4})$, the solution is \textit{ground state}. 
\end{theorem}

\begin{theorem}\label{Teo. exist. sol. eq. Schr. caso zeta=1}
Suppose that $(b_1)-(b_2)$, $(\widehat{b_3})$, $(A_1)-(A_2)$, $(f_1)-(f_2)$, $(\widehat{f_3})$ and $(f_4^*)$ are satisfied. Suppose also that 
\begin{itemize}
\item[$(\widehat{f_6})$] there exists $\widehat{\gamma_0} > 0$ such that 
$$\liminf_{s \rightarrow \infty} \dfrac{sf(s)}{e^{\alpha_0 s^2}} \geq \widehat{\gamma_0} > 4\alpha_0^{-1}\inf_{R>0}\left\{R^{-2}e^{R^2M_R/2}\right\}.$$
\end{itemize}
Then problem $(\widehat{P})$ has a nonnegative nontrivial weak solution. If, in addition, $f$ satisfies $(\widehat{f_4})$, the solution is \textit{ground state}. 
\end{theorem}

To our knowledge there is no paper on Kirchhoff equations in unbounded domains under $(b_1)-(b_3)$, even with nonlinearity having polynomial growth. But on Schr\"odinger equations involving exponential growth, we can cite \cite{Sou,SouDoOSil}. In \cite{Sou}, the author studied the nonhomogeneous singular problem
\begin{equation}\label{eq. Schr. sing. nao-homog.}
-\Delta u + b(x)u = \dfrac{g(x)f(u)}{|x|^a} + h(x), \ \ x \in \mathbb{R}^2,
\end{equation}
with $b$ satisfying $(b_1)-(b_3)$, $f$ having subcritical exponential growth and $a \in [0,2)$. In \cite{SouDoOSil}, the authors studied the nonhomogeneous quasilinear problem
\begin{equation}\label{eq. Schr. quasil. nao-homog.}
-\Delta_N u + b(x)|u|^{N-2}u = c(x)|u|^{N-2}u + g(x)f(u) + \varepsilon h(x), \ \ x \in \mathbb{R}^N,
\end{equation}
where $\Delta_N u = \textrm{div}(|\nabla u|^{N-2}\nabla u)$, $N \geq 2$, with $b$ and $f$ satisfying hypotheses similar to $(b_1),(b_2),(\widehat{b_3})$ and $(f_1),(f_2),(\widehat{f_3}),(f_4^*),(\widehat{f_5})$, respectively. The potential $c$ was taken nonnegative and belonging to an apropriated Lebesgue space, with norm, in this space, bounded by a suitable constant. Notice that, for certain sign-changing potentials $b$, this hypothesis does not include the case in which $b$ is replaced by $b^+$ and $c(x)=b^-(x):=\max\{0,-b(x)\}$ in equation \eqref{eq. Schr. quasil. nao-homog.}. Actually, although $b^+$ satisfies $(b_1),(b_2),(\widehat{b_3})$ whenever $b$ satisfies $(b_1)-(b_3)$, powers of $b^-$ may not be integrable, as for example $b(x)=|x_1x_2|-C$ given previously. For $h \not\equiv 0$ with small norm in an apropriated dual space, two solutions were obtained in \cite{Sou} and \cite{SouDoOSil} for problems \eqref{eq. Schr. sing. nao-homog.} and \eqref{eq. Schr. quasil. nao-homog.}, respectively.

With the potential $b$ satisfying hypotheses similar to \eqref{b coercivo}, we also refer to \cite{LiYang}, forn a Kirchhoff equation, and \cite{DoOMedSev,Yang1}, for Schr\"odinger equations. Other related results can be founded in \cite{AlvSoa,Aou,DoOSouMedSev,DoOSanZha,FeiYin}. On Kirchhoff equations in bounded domains, we refer to \cite{FigSev,GoyMisSre,NaiTar}. All of these papers deal with critical or subcritical exponential growth of Trudinger-Moser type. 

In addition to the aspects already mentioned, our results complement the aforementioned works in other ways: with the exception of \cite{FigSev}, in the other papers it was not proved the existence of \textit{ground state} solutions; differently of \cite{AlvSoa,Aou,DoOMedSev,DoOSouMedSev,DoOSanZha,FeiYin,LiYang,Yang1}, we consider a potential that may change  sign or vanish; in these same papers and in \cite{Sou}, the regularity of the potential is stronger than that considered here; in \cite{Sou,SouDoOSil}, it was assumed that the weight function $g$ in equations \eqref{eq. Schr. sing. nao-homog.} and \eqref{eq. Schr. quasil. nao-homog.} satisfies hypotheses similar to $(A_1)$ and $(A_2)$, but the regularity on $A$ is stronger than here; finally, although in \cite{SouDoOSil} it has been considered a potential $b$ of the same type as ours, the Trudinger-Moser inequality proved here is more general and allow us to consider the more natural hypotheses $(f_6)$ and $(\widehat{f_6})$, instead of $(f_5)$ and $(\widehat{f_5})$.

The rest of this paper is organized as follows: in Section \ref{Resultados preliminares} we prove preliminary results related to Trudinger-Moser inequality; in Section \ref{estr. variac.} we detail the variational framework of problem $(P)$; in Section \ref{estim. minimax} we prove estimates for the Mountain Pass level of the energy functional; finally, in the last section we prove our main results.

\section{Preliminary results}\label{Resultados preliminares}

Hereafter, we  write $\int_{\Omega}u$ instead of $\int_{\Omega} u(x) \textrm{d}x$, for any $\Omega \subset \mathbb{R}^2$ and $u \in L^1(\Omega)$. Norms in $H$, in $W^{1,2}(\mathbb{R}^2)$ and in $L^p(\mathbb{R}^2)$, $1 \leq p \leq \infty$, are denoted by $\left\|\cdot\right\|$, $\left\|\cdot\right\|_{1,2}$ and $\left\|\cdot\right\|_p$, respectively. Notations $C_1, C_2, \dots$ represent positive constants whose exact values are irrelevant. Hypotheses $(b_1)-(b_3)$, $(A_1)-(A_2)$ are always be assumed from now on.

The next result was proved in \cite{DoO} (see also \cite{Cao}).

\begin{lemma}\label{Lema 1}
If $\alpha > 0$ and $v \in W^{1,2}(\mathbb{R}^2)$, then 
$\int_{\mathbb{R}^2} (e^{\alpha v^2} - 1) < \infty.$
Moreover, if $\alpha < 4\pi$, $\left\|\nabla v\right\|_2 \leq 1$ and $\left\|v\right\|_2 \leq M$, then there exists  $C=C(\alpha,M)>0$ such that 
$$\int_{\mathbb{R}^2} (e^{\alpha v^2} - 1) \leq C.$$
\end{lemma}

We need a version of this last result adapted to our variational framework. We start with a technical result.

\begin{lemma}\label{Lema 2}
Let $\beta_0$ be given by hypothesis $(A_2)$ and $\alpha > 0$. For any $v \in H$ and  $r \in [1,\beta_0)$, the function $A(\cdot)^r(e^{\alpha v^2} - 1)^r$ belongs to $L^1(\mathbb{R}^2)$.   
\end{lemma}

\begin{proof} Since $
(e^{\alpha s^2} - 1)^r \leq e^{r\alpha s^2} - 1$, for any $s \in \mathbb{R}$, and $A \in L_{\textrm{loc}}^{\infty}(\mathbb{R}^2)$ we get
\begin{equation}\label{des. 1 Lema 2}
\int_{\mathbb{R}^2} A(x)^r(e^{\alpha v^2} - 1)^r \leq \int_{\mathbb{R}^2 \backslash B_{R_0}(0)} A(x)^r(e^{r\alpha v^2} - 1) + C_1\int_{B_{R_0}(0)} (e^{r\alpha v^2} - 1),
\end{equation}
where $R_0 > 0$ is given by hypothesis $(A_2)$. From Lemma \ref{Lema 1} we conclude that the last integral above is finite. In order to estimate the first one, notice that
\begin{equation}\label{id. 1 Lema 2}
\int_{\mathbb{R}^2 \backslash B_{R_0}(0)} A(x)^r(e^{r\alpha v^2} - 1) = \sum_{m=1}^{\infty} \dfrac{(r\alpha)^m}{m!}\int_{\mathbb{R}^2 \backslash B_{R_0}(0)} A(x)^r v^{2m}.
\end{equation}
Now, by $(A_2)$ and H\"older's inequality, we have that 
\begin{eqnarray}
&& \int_{\mathbb{R}^2 \backslash B_{R_0}(0)} A(x)^r v^{2m} \leq \nonumber \\ 
&& \leq C_2\left\|v\right\|_{2m}^{2m} + C_3\int_{\mathbb{R}^2 \backslash B_{R_0}(0)} (b^+(x))^{r/\beta_0} v^{2m} \label{estim. termo geral serie de taylor} \\
&& \leq C_2\left\|v\right\|_{2m}^{2m} + C_3\left(\int_{\mathbb{R}^2} b^+(x)v^2\right)^{r/\beta_0}\left(\int_{\mathbb{R}^2} v^{2(m\beta_0-r)/(\beta_0-r)}\right)^{(\beta_0-r)/\beta_0}. \nonumber
\end{eqnarray}
But, by $(b_3)$ and $(b_1)$,
\begin{eqnarray*}
	\int_{\mathbb{R}^2} b^+(x)v^2 
	&=& \int_{\mathbb{R}^2} b(x)v^2 - \int_{\{b(x) \leq 0\}} b(x)v^2 
	\leq  \left\|v\right\|^2 + B_0\left\|v\right\|_2^2 \\
	&\leq& \left\|v\right\|^2 + B_0\dfrac{\left\|v\right\|^2}{\lambda_1^b} = C_4\left\|v\right\|^2.
\end{eqnarray*}
This and \eqref{estim. termo geral serie de taylor} imply that
\begin{equation}\label{des. 2 Lema 2}
\begin{array}{lll}
\displaystyle \int_{\mathbb{R}^2 \backslash B_{R_0}(0)} A(x)^r v^{2m} &\leq& C_5\left\|v\right\|^{2m} + C_6\left\|v\right\|^{2r/\beta_0}\left\|v\right\|^{2(m\beta_0-r)/\beta_0} \\
&=& C_7\left\|v\right\|^{2m},
\end{array}
\end{equation}
where we have used that $\min\{2m, \,2(m\beta_0-r)/(\beta_0-r)\} \geq 2$ and $H$ is continuously embbeded into $L^p(\mathbb{R}^2)$, for any $p \geq 2$. Therefore, from \eqref{des. 1 Lema 2}, \eqref{id. 1 Lema 2} and \eqref{des. 2 Lema 2} we obtain
\begin{eqnarray}
\int_{\mathbb{R}^2} A(x)^r(e^{\alpha v^2} - 1)^r 
&\leq& C_7\sum_{m=1}^{\infty} \dfrac{1}{m!}(r\alpha\left\|v\right\|^2)^m + C_1\int_{B_{R_0}(0)} (e^{r\alpha v^2} - 1) \nonumber \\
&=& C_7(e^{r\alpha\left\|v\right\|^2} - 1) + C_1\int_{B_{R_0}(0)} (e^{r\alpha v^2} - 1) \label{des. 3 Lema 2} \\
&<& \infty, \nonumber 
\end{eqnarray}
which completes the proof.
\end{proof}

The following lemma is a version of Lemma \ref{Lema 1} for our framework.

\begin{lemma}\label{Lema 3}
Let $\alpha > 0$, $q>0$ and $\omega, v \in H$. Then 
$$\int_{\mathbb{R}^2} A(x)|\omega|^q(e^{\alpha v^2} - 1) < \infty.$$
Moreover, if $\alpha < 4\pi\zeta^2$ and $\left\|v\right\| \leq 1$, then there exists  $C=C(\alpha,q)>0$ such that 
$$\int_{\mathbb{R}^2} A(x)|\omega|^q(e^{\alpha v^2} - 1) \leq C\left\|\omega\right\|^q.$$
\end{lemma}

\begin{proof} Let $r \in (1,\beta_0)$ be such that $qr'\geq 2$, where $r':=r/(r-1)$. By H\"older's inequality, embbeding $H \hookrightarrow L^{qr'}(\mathbb{R}^2)$ and Lemma \ref{Lema 2} we obtain
\begin{equation}\label{des. 1 Lema 3}
\begin{array}{lll}
\displaystyle \int_{\mathbb{R}^2} A(x)|\omega|^q(e^{\alpha v^2} - 1) 
&\leq& \displaystyle \left\|\omega\right\|_{qr'}^q \left(\int_{\mathbb{R}^2} A(x)^r(e^{\alpha v^2} - 1)^r\right)^{1/r} \\
&\leq& \displaystyle C_1\left\|\omega\right\|^q\left(\int_{\mathbb{R}^2} A(x)^r(e^{\alpha v^2} - 1)^r\right)^{1/r},
\end{array}
\end{equation}
and the first statement is proved.

If $\alpha < 4\pi\zeta^2$ and $\left\|v\right\| \leq 1$, take $r \in (1,\beta_0)$ such that $r\alpha < 4\pi\zeta^2$. By using  \eqref{des. 3 Lema 2}-\eqref{des. 1 Lema 3} and writing $v^2=\zeta^{-2}(\zeta v)^2$, we have that 
$$
\int_{\mathbb{R}^2} A(x)|\omega|^q(e^{\alpha v^2} - 1) \leq C_2\left\|\omega\right\|^q\left(e^{r\alpha\left\|v\right\|^2} - 1 + \int_{B_{R_0}(0)} (e^{r\alpha\zeta^{-2}(\zeta v)^2} - 1)\right)^{1/r}.
$$
Since $\left\|v\right\| \leq 1$, by \eqref{imersao H em D^(1,2)} we have $\left\|\nabla(\zeta  v)\right\|_2 \leq 1$. Furthermore, $\left\|\zeta v\right\|_2 \leq C_3\zeta\left\|v\right\| \leq M$, for some $M > 0$ independent of $v$. The result follows from Lemma \ref{Lema 1}, the above inequality and $r\alpha\zeta^{-2} < 4\pi$. 
\end{proof}

We present now a version of a famous result of Lions \cite[subsection I.7]{PLio}  to our space $H$.

\begin{corollary}\label{Corolario 1}
Let $q > 0$ and let $(\omega_n), (v_n) \subset H$ be such that $(\omega_n)$ is bounded in $H$, $v_n \rightharpoonup v$ weakly in $H$ and $\left\|v_n\right\|=1$, for any $n \in \mathbb{N}$. Then, if $\left\|v\right\| < 1$, for any $0 < p < 4\pi\zeta^2/(1-\left\|v\right\|^2)$ it holds 
$$\sup_{n\in\mathbb{N}}\int_{\mathbb{R}^2} A(x)|\omega_n|^q(e^{pv_n^2} - 1) < \infty.$$
The same holds if $\left\|v\right\|=1$ and $0 < p < \infty$.
\end{corollary}

\begin{proof} First of all notice that, given $a,b \in \mathbb{R}$ and $\varepsilon > 0$, by Young's inequality we have 
\begin{eqnarray*}
a^2 &=& (a-b)^2 + b^2 + 2\varepsilon(a-b)b\varepsilon^{-1} \\
    &\leq& (a-b)^2 + b^2 + 2\left(\dfrac{\varepsilon^2(a-b)^2}{2} + \dfrac{b^2\varepsilon^{-2}}{2}\right) \\
    &=& (1+\varepsilon^2)(a-b)^2 + (1+\varepsilon^{-2})b^2.
\end{eqnarray*}
Thus, if $r_1,r_2 > 1$ are such that $1/r_1 + 1/r_2 = 1$, by using Young's inequality again we obtain
\begin{eqnarray*}
A(x)|\omega_n|^q e^{pv_n^2} &\leq& (A(x)|\omega_n|^q)^{1/r_1}e^{p(1+\varepsilon^2)(v_n - v)^2}(A(x)|\omega_n|^q)^{1/r_2}e^{p(1+\varepsilon^{-2})v^2} \\
							&\leq& \dfrac{1}{r_1}A(x)|\omega_n|^q e^{r_1p(1+\varepsilon^2)(v_n - v)^2} + \dfrac{1}{r_2}A(x)|\omega_n|^q e^{r_2p(1+\varepsilon^{-2})v^2}.
\end{eqnarray*}
So,
\begin{eqnarray*}
\int_{\mathbb{R}^2} A(x)|\omega_n|^q(e^{pv_n^2} - 1)
&\leq& \dfrac{1}{r_1}\int_{\mathbb{R}^2} A(x)|\omega_n|^q\left(e^{r_1p(1+\varepsilon^2)(v_n - v)^2} - 1\right) \\
&& + \ \dfrac{1}{r_2}\int_{\mathbb{R}^2} A(x)|\omega_n|^q\left(e^{r_2p(1+\varepsilon^{-2})v^2} - 1\right).
\end{eqnarray*}
Since $(\omega_n)$ is bounded in $H$, inequality \eqref{des. 1 Lema 3} with $\alpha = r_2p(1+\varepsilon^{-2})$ and Lemma \ref{Lema 2} guarantee that the second integral on the right-hand side above is bounded independently of $n$. In order to estimate the other integral notice that, since $\left\|v_n\right\| = 1$ and $v_n \rightharpoonup v$ weakly in $H$, we get
$$
\lim_{n\rightarrow\infty} p\left\|v_n - v\right\|^2 =  p(1 - \left\|v\right\|^2) < 4\pi\zeta^2.
$$
Then, by taking $r_1>1$ sufficiently close to 1 and $\varepsilon>0$ sufficiently small, there exists $n_0 \in \mathbb{N}$ such that
$$r_1p(1+\varepsilon^2)\left\|v_n - v\right\|^2 < 4\pi\zeta^2, \ \ \forall \ n>n_0.$$
Observing that $(v_n - v)^2=\|v_n - v\|^2((v_n - v)/\|v_n - v\|)^2$, from the above inequality and Lemma \ref{Lema 3} it follows that 
$$
\int_{\mathbb{R}^2} A(x)|\omega_n|^q\left(e^{r_1p(1+\varepsilon^2)(v_n - v)^2} - 1\right) \leq C_1\left\|\omega_n\right\|^q \leq C_2, \ \ \forall \ n>n_0,
$$
which concludes the proof.
\end{proof}

The next result is an easy consequence of the monotonicity conditions $(m_3)$ and $(f_4)$.

\begin{lemma}\label{Lema 9} Suppose that $(m_3)$ and $(f_4)$ hold. Then
\begin{itemize}
	\item[\emph{$(i)$}] the function $L(t) := (1/2)M(t) - (1/4)m(t)t$ is increasing in $[0,\infty)$; in particular, $L(t) > L(0) = 0$, for any $t>0$;
	\item[\emph{$(ii)$}] the function $G(s) := sf(s) - 4F(s)$ is nondecreasing in $[0,\infty)$; in particular, $G(s) \geq G(0)=0$, for any $s>0$.
\end{itemize}
\end{lemma}

\begin{proof} We only prove the first item since the other one is analogous. Let $t_1,t_2 \in \mathbb{R}$ be such that $0<t_1<t_2$. By $(m_3)$, we have
\begin{eqnarray*}
2M(t_1) - m(t_1)t_1 &=& 2M(t_2) - 2\int_{t_1}^{t_2} \dfrac{m(\tau)}{\tau}\tau \ \textrm{d}\tau - \dfrac{m(t_1)}{t_1}t_1^2 \\
                    &<& 2M(t_2) - \dfrac{m(t_2)}{t_2}(t_2^2 - t_1^2) - \dfrac{m(t_2)}{t_2}t_1^2 \\ 
			              &=& 2M(t_2) - m(t_2)t_2.
\end{eqnarray*}
and therefore the function $\widehat{L}(t) = 4L(t)= 2M(t) - m(t)t$ is increasing in $(0,\infty)$. Continuity in $t=0$ implies that this property holds in $[0,\infty)$.  \end{proof}

We finish this section by presenting a convergence result proved in \cite{FigMiyRuf}.

\begin{lemma}\label{Lema 10}
Let $\Omega \subset \mathbb{R}^2$ be a bounded domain. If $f : \overline{\Omega} \times \mathbb{R} \rightarrow \mathbb{R}$ is a continuous function and $(u_n) \subset L^1(\Omega)$ is a sequence such that 
$$ u_n \rightarrow u \ \textrm{in} \ L^1(\Omega), \ \ f(\cdot,u_n),f(\cdot,u) \in L^1(\Omega), \ \ \int_{\Omega} |f(x,u_n)u_n| \leq C,$$
where $C>0$ is a constant, then $f(\cdot,u_n) \rightarrow f(\cdot,u)$ in $L^1(\Omega)$.
\end{lemma}

\section{Variational framework}\label{estr. variac.}

Given $\varepsilon>0$, $\alpha>\alpha_0$ and $q \geq 1$, by $(f_1)$ and $(f_4^*)$ there exists a constant $C=C(\varepsilon,\alpha,q)>0$ such that
\begin{equation} \label{cresc. F}
\max\{|F(s)|,|sf(s)|\} \leq \varepsilon s^2 + C|s|^q(e^{\alpha s^2} - 1),\quad \forall \, s \in \mathbb{R}.
\end{equation}
This, the embbeding $H \hookrightarrow L_A^2(\mathbb{R}^2)$ and Lemma \ref{Lema 3} show that the functional $I:H \rightarrow \mathbb{R}$ given by
\begin{equation}\label{func. energia}
I(u) := \dfrac{1}{2}M(\left\|u\right\|^2) - \int_{\mathbb{R}^2} A(x)F(u), \ \ u \in H,
\end{equation}
is well defined. Moreover, Lemmas \ref{Lema 2}, \ref{Lema 3} and standard arguments show that $I \in C^1(H,\mathbb{R})$ and, for any $u,v \in H$, there holds
\begin{equation}\label{deriv. func. energia}
I'(u)v = m(\left\|u\right\|^2)\int_{\mathbb{R}^2}(\nabla u \cdot \nabla v + b(x)uv) - \int_{\mathbb{R}^2} A(x)f(u)v,
\end{equation}
and therefore critical points of $I$ are precisely the weak solutions of problem $(P)$.

\begin{lemma}\label{Lema 5}
Suppose that $(m_1)$, $(f_1)$ and $(f_4^*)$ hold. Then there exists $\rho > 0$ and $\sigma > 0$ such that $$ I(u) \geq \sigma \ , \quad \forall \ u \in H,\, \left\|u\right\| = \rho.$$
\end{lemma}

\begin{proof} Let $\varepsilon > 0$, $\alpha > \alpha_0$ and $q>2$. By \eqref{cresc. F}, the embbeding $H \hookrightarrow L_A^2(\mathbb{R}^2)$ and Lemma \ref{Lema 3}, if $0 < \rho_1 < (4\pi\zeta^2/\alpha)^{1/2}$, then for $u \in H$ with $\left\|u\right\| \leq \rho_1$ we have that
\begin{eqnarray*}
\int_{\mathbb{R}^2} A(x)F(u) 
&\leq& \varepsilon\int_{\mathbb{R}^2} A(x)u^2 + C\int_{\mathbb{R}^2} A(x)|u|^q(e^{\alpha u^2} - 1) \\
&\leq& \varepsilon C_1\left\|u\right\|^2 + C\int_{\mathbb{R}^2} A(x)|u|^q\left(e^{\alpha\rho_1^2 (u/\left\|u\right\|)^2} - 1\right) \\
&\leq& \varepsilon C_1\left\|u\right\|^2 + C_2\left\|u\right\|^q.
\end{eqnarray*}
Let $m_0>0$ be given by the hypothesis $(m_1)$. Since $M(t) \geq m_0 t$, for any $t \geq 0$, we obtain
$$I(u) \geq \left\|u\right\|^2 \left(\dfrac{m_0}{2} - \varepsilon C_1 - C_2\left\|u\right\|^{q-2}\right),$$
whenever $\left\|u\right\| \leq \rho_1$. Now choose $\varepsilon > 0$ and $0 < \rho \leq \rho_1$ such that $(m_0/2) - \varepsilon C_1 - C_2\rho^{q-2} > 0$. This choice is possible because $q>2$. Thereby, for any $u \in H$ with $\left\|u\right\| = \rho$, we have that $I(u) \geq \sigma$, where 
$$\sigma := \rho^2 \left(\dfrac{m_0}{2} - \varepsilon C_1 - C_2\rho^{q-2}\right) > 0.$$
This concludes the proof. 
\end{proof}

\begin{lemma}\label{Lema 6}
Suppose that $(m_1), (m_3^*), (f_1), (f_3)$ and $(f_4^*)$ hold. If $\rho > 0$ is given by Lemma \ref{Lema 5}, then there exists $v_0 \in H$ such that $I(v_0) < 0$ and $\left\|v_0\right\| > \rho$.  
\end{lemma}

\begin{proof} By the continuity of $m$ and $(m_3^*)$, there exists $a_0>0$ such that
\begin{equation}\label{cresc. M}
M(t) \leq a_0t + a_1\dfrac{t^2}{2}, \ \ \forall \ t \geq 0.
\end{equation}
On the other hand, by $(f_3)$, there exist constants $C_1,C_2>0$ such that 
$$F(s) \geq C_1s^{\theta_0} - C_2, \ \ \forall \ s \geq 0.$$
Now choose $v \in C_0(\mathbb{R}^2) \backslash \{0\}$ with $v \geq 0$ in $\mathbb{R}^2$. If $\Omega \subset \mathbb{R}^2$ contains the support of the function $v$, the above inequalities and $(A_1)$ provide, for any $t \geq 0$,
$$I(t v) \leq a_0t^2\dfrac{\left\|v\right\|^2}{2} + a_1 t^4\dfrac{\left\|v\right\|^4}{4} - C_1 t^{\theta_0}\int_{\Omega} v^{\theta_0} + C_2|\Omega|.$$
Since $\int_{\Omega} v^{\theta_0} > 0$ and $\theta_0 > 4$, we conclude that 
$I(t v) \to -\infty$, as $t \to \infty$. 
Hence the result holds for $v_0 = t_0 v$, with $t_0>0$ large enough. 
\end{proof}

\begin{remark}\label{prova alternativa lema 6}
For future reference we notice that the above lemma can be proved in a different way if $f(s)>0$ for anyl $s>0$. In this case, for any $w \in H$ with $w^+ \not\equiv 0$, we have $\int_{\mathbb{R}^2} A(x)F(w) > 0$. On the other hand, defining, for any $s \in \mathbb{R}$, 
$$\phi_s(t) := t^{-\theta_0}F(t s) - F(s), \ \ t > 0,$$
by $(f_3)$ we have that $\phi_s'(t) \geq 0$, for any $t > 0$. This implies that $\phi_s(t) \geq \phi_s(1) = 0$ for any $t \geq 1$. That is,
$$
F(t s) \geq t^{\theta_0}F(s), \ \ \forall \ t \geq 1.
$$
So, for $t \geq 1$, by \eqref{cresc. M} and the above inequality we have
$$
I(t w) \leq a_0t^2\dfrac{\left\|w\right\|^2}{2} + a_1t^4\dfrac{\left\|w\right\|^4}{4} - t^{\theta_0}\int_{\mathbb{R}^2} A(x)F(w)
$$
and the conclusion follows as before.
\end{remark}

Lemmas \ref{Lema 5} and \ref{Lema 6} show that the energy functional $I$ has the geometry of Mountain Pass Theorem. Thus, there exists a sequence $(u_n) \subset H$ such that 
$$I(u_n) \longrightarrow c^* := \inf_{\gamma\in\Gamma}\max_{t\in[0,1]} I(\gamma(t)) \ \ \textrm{and} \ \ I'(u_n) \longrightarrow 0$$
as $n \rightarrow \infty$, where
 $\Gamma := \{\gamma \in C([0,1],H) : \gamma(0)=0 \ \textrm{and} \ I(\gamma(1))<0\}$. It is worth noticing that, by the definition of $c^*$ and the proof of Lemma \ref{Lema 5}, we easily see that $c^* \geq \sigma > 0$.

\section{Minimax estimates}\label{estim. minimax}

In the first part of this section we will obtain an estimate for $c^*$ in terms of the parameters $\zeta$ and $\alpha_0$, given in the inequality \eqref{imersao H em D^(1,2)} and the hypothesis $(f_1)$, respectively. 

We first consider the case $\zeta<1$ and observe that $S_p$ defined in \eqref{def_sp} is the best constant of the compact embedding $H \hookrightarrow L^p(\mathbb{R}^2)$. Hence, there exists $v_p \in H$ such that $\|v_p\|_p=1$ and $S_p = \left\|v_p\right\| > 0$. Without loss of generality, we may assume that $v_p \geq 0$ a.e. in $\mathbb{R}^2$.

\begin{proposition}\label{estim. nivel PM caso zeta<1}
Suppose that $(m_3^*)$, $(f_1)$, $(f_3)$, $(f_4^*)$ and $(f_5)$ hold. If $\zeta < 1$ then $$c^* < \dfrac{1}{2}M\left(\dfrac{4\pi\zeta^2}{\alpha_0}\right).$$
\end{proposition}

\begin{proof}
 Let $p_0>4$ be given in hypothesis $(f_5)$ and $v_{p_0} \in H $ be such that $\|v_{p_0}\|=S_{p_0}$ and $\|v_{p_0}\|_{p_0}=1$. Recalling that $(f_5)$ implies that $f(s)>0$ for any $s>0$, by Remark \ref{prova alternativa lema 6} we have that $I(tv_{p_0}) \rightarrow -\infty$ as $t \rightarrow \infty$. Thus, from definition of $c^*$ it follows that
$$c^* \leq \max_{t>0} I(tv_{p_0}).$$
By $(A_1)$ and $(f_5)$, 
$$
I(tv_{p_0}) < \dfrac{1}{2}M(t^2\left\|v_{p_0}\right\|^2) - t^{p_0}\dfrac{C_{p_0}}{p_0}\int_{\mathbb{R}^2} |v_{p_0}|^{p_0}, \ \ \forall \ t>0.
$$
Hence, from the definition of $C_{p_0}$ we obtain
$$
\max_{t>0} I(tv_{p_0}) < \max_{t>0}\left\{\dfrac{1}{2}M(t^2S_{p_0}^2) - t^{p_0}\dfrac{C_{p_0}}{p_0}\right\} \leq \dfrac{1}{2}M\left(\dfrac{4\pi\zeta^2}{\alpha_0}\right), 
$$
which concludes the proof. 
\end{proof}

In order to deal with the case $\zeta=1$ we define, for $n\geq 2$ and $R>0$, the following sequence of scaled and truncated Green's functions (see Moser \cite{Mos}):
\begin{equation}\label{funcoes de Green}
\widetilde{G}_n(x) = \dfrac{1}{\sqrt{2\pi}} \left\{\begin{array}{ll} (\log n)^{1/2} ,& \textrm{if} \ |x| \leq R/n, \vspace{0.2cm} \\  \dfrac{\log(R/|x|)}{(\log n)^{1/2}} ,& \textrm{if} \ R/n \leq |x| \leq R, \vspace{0.2cm} \\ 0 ,& \textrm{if} \ |x| \geq R. \end{array} \right. 
\end{equation}
Notice that $\widetilde{G}_n \in W^{1,2}(\mathbb{R}^2)$ and $\textrm{supp}(\widetilde{G}_n) = \overline{B_R(0)}$. Consequently, $\widetilde{G}_n \in H$. Furthermore, 
$$
\int_{\mathbb{R}^2} |\nabla \widetilde{G}_n|^2 = \dfrac{1}{2\pi \log n}\int_{\left\{R/n<|x|<R\right\}} |x|^{-2} = \dfrac{1}{\log n}\int_{R/n}^R s^{-1} \textrm{d}s = 1 
$$
and, recalling the notation $M_R = \left\|b\right\|_{L^{\infty}(B_R(0))}$,
\begin{eqnarray*}
\int_{\mathbb{R}^2} b(x)|\widetilde{G}_n|^2 
&=& \dfrac{\log n}{2\pi} \int_{B_{R/n}(0)} b(x) + \dfrac{1}{2\pi \log n}\int_{\left\{R/n \leq |x| \leq R\right\}} b(x)\log^2\left(\dfrac{R}{|x|}\right) \\
&\leq& \dfrac{R^2M_R \log n}{2n^2} + \dfrac{M_R}{\log n}\int_{R/n}^R s\log^2\left(\dfrac{R}{s}\right) \textrm{d}s \\
&=& \dfrac{R^2M_R \log n}{2n^2} + \dfrac{R^2M_R}{\log n}\left(\dfrac{n^2-1}{4n^2} - \dfrac{\log^2(n) + \log n}{2n^2}\right) \\
&\leq& \dfrac{R^2M_R}{4\log n}. 
\end{eqnarray*}
Then, by denoting $\xi_n := \|\widetilde{G}_n\|$, we have $\xi_n^2 \leq 1 + R^2M_R/(4\log n)$ and $\xi_n \rightarrow 1$ as $n \rightarrow \infty$. 

We now  consider the sequence of functions 
$$G_n := \dfrac{\widetilde{G}_n}{\xi_n}$$  and prove the following technical result:

\begin{lemma}\label{Lema 7}
We have that $$\liminf_{n\rightarrow\infty} \int_{B_R(0)} e^{4\pi G_n^2} \geq \pi R^2 e^{-R^2M_R/2} + \pi R^2.$$
\end{lemma}

\begin{proof}
Since $\xi_n^2 \leq 1 + R^2M_R/(4\log n)$, then
$$
2(\xi_n^{-2} - 1) \log n = 2\xi_n^{-2}(1 - \xi_n^2) \log n \geq -\xi_n^{-2}\dfrac{R^2M_R}{2}
$$
and therefore
\begin{equation}\label{des. bola raio R/n}
\begin{array}{lcl}
\displaystyle\int_{B_{R/n}(0)} e^{4\pi G_n^2} &=& \displaystyle\int_{B_{R/n}(0)} e^{2\xi_n^{-2}\log n} \vspace{0.2cm}\\
&=& \pi R^2e^{2(\xi_n^{-2} - 1)\log n} \geq \pi R^2e^{-\xi_n^{-2}R^2M_R/2}.
\end{array}
\end{equation}

On the other hand, by using the change of variable $t = \xi_n^{-1}\log(R/s)/\log n$, we get
\begin{eqnarray*}
\int_{\left\{R/n \leq |x| \leq R\right\}} e^{4\pi G_n^2} 
&=& \int_{\left\{R/n \leq |x| \leq R\right\}} e^{2\xi_n^{-2}\log^2(R/|x|)/\log n} \\
&=& 2\pi\int_{R/n}^R se^{2(\xi_n^{-1}\log(R/s)/\log n)^2 \log n} \textrm{d}s \\
&=& 2\pi R^2 \xi_n \log n \int_0^{\xi_n^{-1}} e^{2(t^2 - \xi_n t)\log n} \textrm{d}t \\
&\geq& 2\pi R^2 \xi_n \log n \int_0^{\xi_n^{-1}} e^{-2\xi_n t\log n} \textrm{d}t \\
&=& -\pi R^2 e^{-2\log n} + \pi R^2.
\end{eqnarray*}
Therefore, since $\displaystyle \lim_{n \rightarrow \infty}\xi_n = 1$, it follows from \eqref{des. bola raio R/n} and the above inequality  that
$$
\liminf_{n\rightarrow\infty} \int_{B_R(0)} e^{4\pi G_n^2} \geq \pi R^2 e^{-R^2M_R/2} + \pi R^2,
$$
as stated. 
\end{proof}

Now, for $\zeta=1$, we can use the previous lemma to obtain the same estimate of Proposition \ref{estim. nivel PM caso zeta<1} with condition $(f_6)$ instead of $(f_5)$:

\begin{proposition}\label{estim. nivel PM}
Suppose that $(m_3^*)$, $(f_1)$, $(f_3)$, $(f_4^*)$ and $(f_6)$ hold. Then $$c^* < \dfrac{1}{2}M\left(\dfrac{4\pi}{\alpha_0}\right).$$
\end{proposition}

\begin{proof}
As in the proof of Lemma \ref{Lema 6}, we have that $I(tG_n) \rightarrow -\infty$ as $t \rightarrow \infty$. By definition of $c^*$, it follows that
$$c^* \leq \max_{t>0} I(tG_n),\quad \forall\, n \geq 2.$$
Since the functional $I$ has the Mountain Pass geometry, for each $n$ there exists $t_n>0$ such that
$$I(t_nG_n) = \max_{t>0}I(tG_n).$$
Thus, it is enough to prove that, for some $n \in \mathbb{N}$, we have
$$I(t_nG_n) < \dfrac{1}{2}M\left(\dfrac{4\pi}{\alpha_0}\right).$$

Suppose, by contradiction, that the above inequality is false. Since $\left\|G_n\right\|=1$, we have that 
$$I(t_nG_n) = \dfrac{1}{2}M(t_n^2) - \int_{\mathbb{R}^2} A(x)F(t_nG_n) \geq \dfrac{1}{2}M\left(\dfrac{4\pi}{\alpha_0}\right),\quad \forall\, n \geq 2.$$
Since $A$ and $F$ are nonnegative, this implies that $M(t_n^2) \geq M(4\pi/\alpha_0)$. But $M$ is a increasing function, because its derivative $m$ is positive. We conclude that
\begin{equation}\label{cota inferior tn}
t_n^2 \geq \dfrac{4\pi}{\alpha_0}.
\end{equation}

On the other hand, since $I'(t_nG_n)t_nG_n = 0$, we can use $(A_1)$, $f \geq 0$ and $\textrm{supp}(G_n) = \overline{B_R(0)}$ to obtain
\begin{equation} \label{ident. deriv. em tn}
\begin{array}{lcl}
m(t_n^2)t_n^2 &=& \displaystyle\int_{B_R(0)} A(x)f(t_nG_n)t_nG_n  \vspace{0.2cm}\\
              &\geq& \displaystyle\int_{B_{R/n}(0)} f(t_nG_n)t_nG_n  \vspace{0.2cm}\\
					    &=& \displaystyle\int_{B_{R/n}(0)} f\left(\dfrac{t_n\xi_n^{-1}}{\sqrt{2\pi}}(\log n)^{1/2}\right)\dfrac{t_n\xi_n^{-1}}{\sqrt{2\pi}}(\log n)^{1/2}. 
\end{array}
\end{equation}
But notice that, given $0<\delta<\gamma_0$, by $(f_6)$ there exists $s_{\delta}>0$ such that 
\begin{equation}\label{conseq. (f7)}
f(s)s \geq (\gamma_0 - \delta)e^{\alpha_0 s^2}, \ \ \forall \ s \geq s_{\delta}.
\end{equation}
Since $t_n\xi_n^{-1}(\log n)^{1/2} \rightarrow \infty$ as $n \rightarrow \infty$, because $\xi_n \rightarrow 1$ and $t_n \not\rightarrow 0$, it follows that, for $n$ large,
\begin{eqnarray*}
m(t_n^2)t_n^2 &\geq& \int_{B_{R/n}(0)} (\gamma_0 - \delta)e^{\alpha_0 t_n^2(\xi_n\sqrt{2\pi})^{-2}\log n} \\
			  &=& \pi R^2(\gamma_0 - \delta)e^{\left(\alpha_0 t_n^2(\xi_n\sqrt{2\pi})^{-2} - 2\right)\log n}.
\end{eqnarray*}
This inequality and $(m_3^*)$ imply that the sequence $(t_n) \subset (0,\infty)$ is bounded and, consequently, there exists $t_0 > 0$ such that, up to a subsequence, $t_n \rightarrow t_0$ as $n \rightarrow \infty$. In this case, the above inequality also implies that 
$$
\lim_{n \rightarrow \infty} \left(\alpha_0 t_n^2(\xi_n\sqrt{2\pi})^{-2} - 2\right) = 2\left(\dfrac{\alpha_0}{4\pi}t_0^2 - 1\right) \leq 0.
$$
From this and \eqref{cota inferior tn}, we infer  that
\begin{equation}\label{converg. tn}
\lim_{n\to \infty} t_n^2 = \dfrac{4\pi}{\alpha_0}.
\end{equation} 

Now, for each $n \geq 2$, define the sets
$$D_{n,\delta} := \left\{x \in B_R(0) : t_nG_n(x) \geq s_{\delta}\right\}, \ \ E_{n,\delta} := B_R(0) \backslash D_{n,\delta}.$$
By hypothesis $(A_1)$, \eqref{ident. deriv. em tn} and \eqref{conseq. (f7)}, we have that 
\begin{eqnarray}
m(t_n^2)t_n^2 &\geq& \int_{D_{n,\delta}} f(t_nG_n)t_nG_n + \int_{E_{n,\delta}} f(t_nG_n)t_nG_n \nonumber \\
              &\geq& (\gamma_0 - \delta)\left(\int_{B_R(0)} e^{\alpha_0 t_n^2G_n^2} - \int_{E_{n,\delta}} e^{\alpha_0 t_n^2G_n^2}\right)  \label{estim. tn integral truncada} \\ 
							&& + \int_{E_{n,\delta}} f(t_nG_n)t_nG_n. \nonumber
\end{eqnarray}
But $G_n(x) \rightarrow 0$ for a.e. $x \in B_R(0)$ and, therefore, $\chi_{E_{n,\delta}}(x) \rightarrow 1$ for a.e. $x \in B_R(0)$, as $n \rightarrow \infty$, where $\chi_{E_{n,\delta}}$ is the characteristic function of $E_{n,\delta}$. Moreover, $t_nG_n < s_{\delta}$ in $E_{n,\delta}$. Then, it follows from the Lebesgue's Theorem that
$$\int_{E_{n,\delta}} e^{\alpha_0 t_n^2G_n^2} \longrightarrow \pi R^2, \ \ \int_{E_{n,\delta}} f(t_nG_n)t_nG_n \longrightarrow 0.$$
Hence, by \eqref{cota inferior tn}, \eqref{converg. tn}, \eqref{estim. tn integral truncada} and Lemma \ref{Lema 7}, we get
\begin{eqnarray*}
m\left(\dfrac{4\pi}{\alpha_0}\right)\dfrac{4\pi}{\alpha_0} 
&\geq& (\gamma_0 - \delta)\liminf_{n \rightarrow \infty}\left(\int_{B_R(0)} e^{\alpha_0 t_n^2G_n^2}\right) - (\gamma_0 - \delta)\pi R^2 \\
&\geq& (\gamma_0 - \delta)\liminf_{n \rightarrow \infty}\left(\int_{B_R(0)} e^{4\pi G_n^2}\right) - (\gamma_0 - \delta)\pi R^2 \\ 
&\geq& (\gamma_0 - \delta)\pi R^2e^{-R^2M_R/2}.
\end{eqnarray*}
Since $0<\delta<\gamma_0$ is arbitrary, we can let $\delta \rightarrow 0^+$ in the above inequality to obtain 
$$\gamma_0 \leq \dfrac{4}{\alpha_0} m\left(\dfrac{4\pi}{\alpha_0}\right)R^{-2}e^{R^2M_R/2}.$$
Since $R>0$ is also arbitrary, we can take the infimum for $R>0$ in this inequality and obtain a contradiction with $(f_6)$. This concludes the proof. \end{proof}

Let $\mathcal{N}$ be the Nehari manifold associated to the functional $I$, namely $$\mathcal{N} := \{u \in H \setminus \{0\} : I'(u)u = 0\}$$ 
and define
$$
d^* := \inf_{u \in \mathcal{N}} I(u).
$$
The next result shows that obtaining a ground state solution is equivalent to show that there exists a critical point $u_0$ such that $I(u_0) = c^*$. 

\begin{lemma}\label{Lema 8}
Suppose that $(m_3)$, $(f_1)$, $(f_3)$ and $(f_4)$ hold. Then $c^* \leq d^*.$
\end{lemma}

\begin{proof}
 Let $u \in \mathcal{N}$. Then, recalling that $f(s) = 0$ for $s \leq 0$, the fact that $u \neq 0$ and $I'(u)u = 0$ implies that $u^+ \not\equiv 0$. If $h(t) := I(tu)$, $t \geq 0$, we have \begin{eqnarray*}
h'(t) = I'(tu)u 
&=& I'(tu)u - t^3I'(u)u \\
&=& m(t^2\left\|u\right\|^2)t\left\|u\right\|^2 - \int_{\mathbb{R}^2} A(x)f(tu)u \\ 
&& - \ t^3 m(\left\|u\right\|^2)\left\|u\right\|^2 + t^3\int_{\mathbb{R}^2} A(x)f(u)u \\
&=& t^3\left\|u\right\|^4\left(\dfrac{m(t^2\left\|u\right\|^2)}{t^2\left\|u\right\|^2} - \dfrac{m(\left\|u\right\|^2)}{\left\|u\right\|^2}\right) \\
&& + \ t^3 \int_{\{u>0\}} A(x)u^4\left(\dfrac{f(u)}{u^3} - \dfrac{f(tu)}{(tu)^3}\right),
\end{eqnarray*}
for any $t>0$. Thus, by $(m_3)$ and $(f_4)$, we have that $h'(t) \geq 0$ for $0<t<1$ and $h'(t) \leq 0$ for $t>1$. Since $h'(1) = I'(u)u = 0$, then
$$I(u) = h(1) = \max_{t \geq 0} h(t) = \max_{t \geq 0} I(tu).$$
On the other hand, since $u^+ \not\equiv 0$ and $(f_4)$ implies that $f(s)>0$ for any $s>0$, by Remark \ref{prova alternativa lema 6} there exists $t_0 > 0$ such that $I(t_0u)<0$. Defining $\gamma : [0,1] \rightarrow H$ by $\gamma(t) := tt_0u$, from definition of $c^*$ it follows that
$$c^* \leq \max_{t \in [0,1]} I(\gamma(t)) \leq \max_{t \geq 0} I(tu) = I(u).$$
Since $u \in \mathcal{N}$ is arbitrary, we conclude that $c^* \leq d^*$.
\end{proof}

\section{Proof of the main theorems}

We present in this final section the proofs for our main theorems. We first prove that Palais-Smale sequences are bounded.
 
\begin{proposition}\label{Proposicao seq. de Palais-Smale}
Suppose that $(m_1)$, $(m_3)$, $(f_1)-(f_3)$ and $(f_4^*)$ hold. Let $(u_n) \subset H$ be a Palais-Smale sequence for the functional $I$ in the level $c \in \mathbb{R}$, that is, 
$$I(u_n) \longrightarrow c \ \ \textrm{and} \ \ I'(u_n) \longrightarrow 0$$
as $n \rightarrow \infty$. Then $(u_n)$ is bounded in $H$. Moreover, up to a subsequence,
\begin{itemize}
	\item[\emph{$(i)$}] $\displaystyle \int_{\Omega} A(x)f(u_n) \longrightarrow \int_{\Omega} A(x)f(u)$, for any bounded domain $\Omega \subset \mathbb{R}^2$;
	\item[\emph{$(ii)$}] $\displaystyle \int_{\mathbb{R}^2} A(x)F(u_n) \longrightarrow \int_{\mathbb{R}^2} A(x)F(u).$
\end{itemize}
\end{proposition}

\begin{proof}
By using Lemma \ref{Lema 9}$(i)$ and $(f_3)$, we get
$$
 c + o(1) + \left\|u_n\right\| \geq    I(u_n) - \dfrac{1}{\theta_0}I'(u_n)u_n  \geq  \left(\dfrac{\theta_0 - 4}{4\theta_0}\right)m_0\left\|u_n\right\|^2,
$$
as $n\to \infty$, where $m_0$ is given in hypothesis $(m_1)$. Since $\theta_0>4$ and $m_0>0$, the above inequality implies that the sequence $(u_n)$ is bounded in $H$.

Let $\Omega \subset \mathbb{R}^2$ be a bounded domain. Since $u_n \rightharpoonup u$ weakly in $H$, it follows that $u_n \rightarrow u$ in $L^1(\Omega)$, up to a subsequence. Moreover, since $I'(u_n)u_n \rightarrow 0$ as $n \rightarrow \infty$, we get
\begin{equation}\label{estim. norma L1 Af(u_n)u_n)}
\int_{\Omega} |f(u_n)u_n| \leq \int_{\mathbb{R}^2} A(x)f(u_n)u_n = m(\left\|u_n\right\|^2)\left\|u_n\right\|^2 - I'(u_n)u_n \leq C_1.
\end{equation}
By \eqref{cresc. F}, $f(u_n),\,f(u) \in L^1(\Omega)$ and therefore we conclude from Lemma \ref{Lema 10} that $f(u_n) \rightarrow f(u)$ in $L^1(\Omega)$. But
$$\int_{\Omega} A(x)|f(u_n) - f(u)| \leq \left\|A\right\|_{L^{\infty}(\Omega)}\int_{\Omega} |f(u_n) - f(u)| \longrightarrow 0,$$
which proves $(i)$. For the second item we take $r>0$ and use $(i)$ to obtain $h \in L^1(B_r(0))$ such that $A(x)f(u_n(x)) \leq h(x)$ for a.e. $x \in B_r(0)$. So, by using $(f_2)$ we get
\begin{eqnarray*}
A(x)F(u_n(x)) &\leq& \left\|A\right\|_{L^{\infty}(B_r(0))}\max_{s \in [0,s_0]}F(s) + K_0A(x)f(u_n(x)) \\
              &\leq& \left\|A\right\|_{L^{\infty}(B_r(0))}F(s_0) + K_0h(x)
\end{eqnarray*}
for a.e. $x \in B_r(0)$. Since we may assume that $u_n(x) \rightarrow u(x)$ for a.e. $x \in \mathbb{R}^2$ and $F$ is continuous, by Lebesgue's Theorem we obtain 
$$\int_{B_r(0)} A(x)F(u_n) \longrightarrow \int_{B_r(0)} A(x)F(u).$$
Thus, in order to conclude the proof of item $(ii)$, it is enough to show that, given $\delta>0$, there exists $r>0$ such that:
\begin{equation}\label{integrais AF(.) fora da bola}
\int_{\mathbb{R}^2 \backslash B_r(0)} A(x)F(u_n) < \delta, \ \ \forall \ n \in \mathbb{N}; \ \ \int_{\mathbb{R}^2 \backslash B_r(0)} A(x)F(u) < \delta.
\end{equation}
Since  $A(\cdot)F(u)$ is integrable, the second inequality holds for $r>0$ large. For the first one, we can use $(f_2)$ and $(f_4^*)$ to write
$$F(s) \leq C_2|s|^2 + C_3f(s), \ \ \forall \ s \in \mathbb{R}.$$
Then, given $K>0$, by the above inequality, the embbeding $H \hookrightarrow L_A^3(\mathbb{R}^2)$, the boundedness of $(u_n)$ in $H$ and \eqref{estim. norma L1 Af(u_n)u_n)}, we have that
\begin{eqnarray*}
\int_{\{|u_n|>K\} \cap (\mathbb{R}^2 \backslash B_r(0))} A(x)F(u_n) &\leq& C_2\int_{\{|u_n|>K\} \cap (\mathbb{R}^2 \backslash B_r(0))} A(x)|u_n|^2 \\
&& + \ C_3\int_{\{|u_n|>K\} \cap (\mathbb{R}^2 \backslash B_r(0))} A(x)f(u_n) \\
&\leq& \dfrac{C_2}{K}\int_{\mathbb{R}^2} A(x)|u_n|^3  +  \dfrac{C_3}{K}\int_{\mathbb{R}^2} A(x)f(u_n)u_n \\
&\leq& \dfrac{C_4}{K}.
\end{eqnarray*}
Thus, we can choose $K$ large enough such that 
$$\int_{\{|u_n|>K\} \cap (\mathbb{R}^2 \backslash B_r(0))} A(x)F(u_n) < \dfrac{\delta}{2}, \ \ \forall \ n \in \mathbb{N}.$$
On the other hand, by inequality \eqref{cresc. F} with $q=2$, for $|s| \leq K$ we have that
\begin{eqnarray*}
F(s) \leq C_5|s|^2 + C_6|s|^2(e^{\alpha s^2} - 1) \leq \left(C_5 + C_6(e^{\alpha K^2} - 1)\right)|s|^2 \leq C_7|s|^2,
\end{eqnarray*}
where $\alpha>\alpha_0$ and $C_7=C_7(\alpha,K)>0$ are constants. Then
$$
\int_{\{|u_n| \leq K\} \cap (\mathbb{R}^2 \backslash B_r(0))} A(x)F(u_n) \leq C_7\int_{\{|u_n| \leq K\} \cap (\mathbb{R}^2 \backslash B_r(0))} A(x)|u_n|^2.
$$
Since $u_n \rightarrow u$ in $L_A^2(\mathbb{R}^2)$, there exists $g \in L^1(\mathbb{R}^2)$ such that $A(x)|u_n(x)|^2 \leq g(x)$ for a.e. $x \in \mathbb{R}^2$. So, by choosing $r>0$ large enough such that $C_7\int_{\mathbb{R}^2 \backslash B_r(0)} g(x) < \delta/2$, we have 
$$\int_{\{|u_n| \leq K\} \cap (\mathbb{R}^2 \backslash B_r(0))} A(x)F(u_n) < \dfrac{\delta}{2}, \ \ \forall \ n \in \mathbb{N}.$$
Combining the above estimates, we obtain \eqref{integrais AF(.) fora da bola}, which concludes the proof of the second item. 
\end{proof}

We are ready to prove Theorems \ref{Teorema exist. sol. caso zeta<1} and \ref{Teorema exist. sol. caso zeta=1}. \\

\noindent \textit{Proof of Theorem \ref{Teorema exist. sol. caso zeta<1}.} As previously observed there exists $(u_n) \subset H$ such that
\begin{equation}\label{seq. P-S nivel P.M.}
I(u_n) \longrightarrow c^* \ \ \textrm{and} \ \ I'(u_n) \longrightarrow 0,
\end{equation}
as $n \rightarrow \infty$. By Proposition \ref{Proposicao seq. de Palais-Smale}, this sequence is bounded in $H$ and therefore we may assume that, for some $u_0 \in H$, 
\begin{equation}\label{conv. seq. P-S nivel P.M.}
u_n \rightharpoonup u_0 \textrm{ weakly in } H, \ \ u_n \rightarrow u_0 \textrm{ in } L_A^2(\mathbb{R}^2).
\end{equation}

We claim that
\begin{equation}\label{I(u_0) nao-negativo}
I(u_0) \geq 0.
\end{equation}
Indeed, suppose by contradiction that $I(u_0) < 0$. Then $u_0 \neq 0$ and, defining $h(t) := I(tu_0)$, $t \geq 0$, we have that $h(0)=0$ and $h(1)<0$. Arguing as in the proof of Lemma \ref{Lema 5} we see that $h(t)>0$, for any $t>0$ small. Thus,  there exists $t_0 \in (0,1)$ such that 
$$
h(t_0) = \max_{t \in [0,1]} h(t) = \max_{t \in [0,1]} I(tu_0), \ \ h'(t_0) = I'(t_0u_0)u_0 = 0.
$$
So, by definition of $c^*$ and Lemma \ref{Lema 9},
\begin{eqnarray*}
c^* \leq h(t_0) 
&=& h(t_0) - \dfrac{1}{4}h'(t_0)t_0 \\
&=& \dfrac{1}{2}M(\left\|t_0u_0\right\|^2) - \dfrac{1}{4}m(\left\|t_0u_0\right\|^2)\left\|t_0u_0\right\|^2 \\
&& + \ \dfrac{1}{4}\int_{\mathbb{R}^2} A(x)\left(f(t_0u_0)t_0u_0 - 4F(t_0u_0)\right) \\
&<& \dfrac{1}{2}M(\left\|u_0\right\|^2) - \dfrac{1}{4}m(\left\|u_0\right\|^2)\left\|u_0\right\|^2 \\
&& + \ \dfrac{1}{4}\int_{\mathbb{R}^2} A(x)\left(f(u_0)u_0 - 4F(u_0)\right).
\end{eqnarray*}
From this inequality, the lower semicontinuity of the norm, Fatou's Lemma and \eqref{seq. P-S nivel P.M.}, it follows that
\begin{eqnarray*}
c^* 
&<& \liminf_{n\rightarrow\infty} \left(\dfrac{1}{2}M(\left\|u_n\right\|^2) - \dfrac{1}{4}m(\left\|u_n\right\|^2)\left\|u_n\right\|^2\right) \\
&& + \ \dfrac{1}{4}\liminf_{n\rightarrow\infty} \int_{\mathbb{R}^2} A(x)\left(f(u_n)u_n - 4F(u_n)\right) \\
&\leq& \liminf_{n\rightarrow\infty} \left(I(u_n) - \dfrac{1}{4}I'(u_n) \right)=c^*,
\end{eqnarray*} 
which is absurd. Therefore, inequality \eqref{I(u_0) nao-negativo} holds. 

Now we will show that $I'(u_0) = 0$ and $I(u_0) = c^*$. Let $\rho_0 \geq 0$ such that $\left\|u_n\right\| \rightarrow \rho_0$. Clearly $\left\|u_0\right\| \leq \rho_0$ and we shall prove that the equality holds. Suppose, by contradiction, that $\left\|u_0\right\| < \rho_0$. Defining $v_n := u_n/\left\|u_n\right\|$ and $v_0 := u_0/\rho_0$, we have that $v_n \rightharpoonup v_0$ weakly in $H$ and $\left\|v_0\right\|<1$. So, by Corollary \ref{Corolario 1}, it follows that  
\begin{equation}\label{versao des. Tru-Mos para seq. v_n}
\sup_n\int_{\mathbb{R}^2} A(x)|u_n - u_0|^q(e^{pv_n^2} - 1) < \infty, \ \ \forall \ q>0, \ \ \forall \ p < \dfrac{4\pi\zeta^2}{1-\left\|v_0\right\|^2}.
\end{equation}
On the other hand, by using \eqref{seq. P-S nivel P.M.}, Proposition \ref{Proposicao seq. de Palais-Smale}$(ii)$, Proposition \ref{estim. nivel PM caso zeta<1}, \eqref{I(u_0) nao-negativo} and hypothesis $(m_2)$, we have that
\begin{eqnarray*}
M(\rho_0^2) &=& \displaystyle \lim_{n\rightarrow\infty} M(\left\|u_n\right\|^2) = \displaystyle \lim_{n\rightarrow\infty} 2\left(I(u_n) + \int_{\mathbb{R}^2} A(x)F(u_n)\right) \\
			&=& \displaystyle 2c^* + 2\int_{\mathbb{R}^2} A(x)F(u_0) = 2c^* + M(\left\|u_0\right\|^2) - 2I(u_0) \\
			&<& M\left(\dfrac{4\pi\zeta^2}{\alpha_0}\right) + M(\left\|u_0\right\|^2) 
			\leq M\left(\dfrac{4\pi\zeta^2}{\alpha_0} + \left\|u_0\right\|^2\right).
\end{eqnarray*}
Since $M$ is increasing, it follows that $\rho_0^2 < (4\pi\zeta^2/\alpha_0) + \left\|u_0\right\|^2$. Hence, by observing that $\rho_0^2 = (\rho_0^2 - \left\|u_0\right\|^2)/(1 - \left\|v_0\right\|^2)$, we get
$$\alpha_0\rho_0^2 < \dfrac{4\pi\zeta^2}{1 - \left\|v_0\right\|^2}.$$
Then, there exists $\eta>0$ such that $\alpha_0\left\|u_n\right\|^2 < \eta < 4\pi\zeta^2/(1 - \left\|v_0\right\|^2)$ for any $n$ large enough. Thus, we can choose $r \in (1,2)$ close to 1 and $\alpha>\alpha_0$ close to $\alpha_0$ such that we still have $r\alpha\left\|u_n\right\|^2 < \eta < 4\pi\zeta^2/(1 - \left\|v_0\right\|^2)$ and, by \eqref{versao des. Tru-Mos para seq. v_n},
\begin{eqnarray*}
\int_{\mathbb{R}^2} A(x)|u_n - u_0|^{2-r}(e^{r\alpha u_n^2} - 1) 
&=& \int_{\mathbb{R}^2} A(x)|u_n - u_0|^{2-r}(e^{r\alpha\left\|u_n\right\|^2 v_n^2} - 1) \\ 
&\leq& \int_{\mathbb{R}^2} A(x)|u_n - u_0|^{2-r}(e^{\eta v_n^2} - 1) \leq C_1,
\end{eqnarray*}
for any $n$ large. Therefore, by using inequality \eqref{cresc. F} with $q=1$, H\"older's inequality, $H \hookrightarrow L_A^2(\mathbb{R}^2)$, Lemma \ref{Lema 2}$(i)$ and \eqref{conv. seq. P-S nivel P.M.}, we obtain 
\begin{eqnarray*}
&& \left|\int_{\mathbb{R}^2} A(x)f(u_n)(u_n - u_0)\right| \leq \\
&& \leq C_2\int_{\mathbb{R}^2} A(x)|u_n||u_n-u_0| + \ C_3\int_{\mathbb{R}^2} A(x)|u_n - u_0|(e^{\alpha u_n^2} - 1) \\ 
&& = C_2\int_{\mathbb{R}^2} \sqrt{A(x)}|u_n|\sqrt{A(x)}|u_n-u_0| \\ 
&& \ \ \ \ + \ C_3\int_{\mathbb{R}^2} (A(x)|u_n - u_0|^2)^{(r-1)/r}(A(x)|u_n - u_0|^{2-r})^{1/r}(e^{\alpha u_n^2} - 1) \\
&& \leq C_4\left\|u_n\right\|\left\|u_n-u_0\right\|_{L_A^2(\mathbb{R}^2)} \\ 
&& \ \ \ \ + \ C_3\left\|u_n-u_0\right\|_{L_A^2(\mathbb{R}^2)}^{2(r-1)/r}\left(\int_{\mathbb{R}^2}A(x)|u_n - u_0|^{2-r}(e^{r\alpha u_n^2} - 1)\right)^{1/r} \\
&& \leq C_5\left\|u_n-u_0\right\|_{L_A^2(\mathbb{R}^2)} + C_6\left\|u_n-u_0\right\|_{L_A^2(\mathbb{R}^2)}^{2(r-1)/r} \longrightarrow 0,
\end{eqnarray*}
as $n \rightarrow \infty$. Since $I'(u_n)(u_n-u_0) \rightarrow 0$ as $n \rightarrow \infty$, we conclude that 
\begin{eqnarray*}
0 &=& \lim_{n \rightarrow \infty} \left(I'(u_n)(u_n-u_0) + \int_{\mathbb{R}^2} A(x)f(u_n)(u_n - u_0)\right) \\
  &=& \lim_{n \rightarrow \infty} m(\left\|u_n\right\|^2)\left\langle u_n,u_n - u_0\right\rangle_H \\
  &=& m(\rho_0^2)(\rho_0^2 - \left\|u_0\right\|^2) \\
  &>& 0,
\end{eqnarray*}
which does not make sense. Thus, we have that $\left\|u_0\right\| = \rho_0 = \lim_{n\rightarrow\infty} \left\|u_n\right\|$ and therefore $u_n \rightarrow u_0$ strongly in $H$. Since $I \in C^1(H,\mathbb{R})$, from \eqref{seq. P-S nivel P.M.} we conclude that $I(u_0) = c^* \neq 0$ and $I'(u_0) = 0$. Recalling that $f(s)=0$, for $s \leq 0$, we can use Lemma \ref{Lema 8} to conclude that $u_0 \geq 0$ is a ground state solution. \hfill $\Box$ \\

\noindent \textit{Proof of Theorem \ref{Teorema exist. sol. caso zeta=1}.} It is sufficient to argue as in the the proof of Theorem \ref{Teorema exist. sol. caso zeta<1}, considering now $\zeta=1$ and using Proposition \ref{estim. nivel PM} instead of Proposition \ref{estim. nivel PM caso zeta<1}. \hfill $\Box$ \bigskip

From now on we suppose that $m \equiv 1$. Hence, the equation in $(P)$ becomes the Schr\"odinger equation
$$
-\Delta u + b(x)u = A(x)f(u) \ \ \ \textrm{in} \ \ \ \mathbb{R}^2.
\leqno{(\widehat{P})}
$$
The energy functional associated to this problem is given by
\begin{equation}\label{func. energia eq. Schr.}
J(u) := \dfrac{1}{2}\left\|u\right\|^2 - \int_{\mathbb{R}^2} A(x)F(u), \ \ u \in H.
\end{equation}
Under hypotheses $(f_1)$, $(\widehat{f_3})$ and $(f_4^*)$, we can prove  that $J \in C^1(H,\mathbb{R})$,
\begin{equation}\label{deriv. func. energia eq. Schr.}
J'(u)v = \int_{\mathbb{R}^2}(\nabla u \cdot \nabla v + b(x)uv) - \int_{\mathbb{R}^2} A(x)f(u)v, \ \ \forall \ u,v \in H,
\end{equation}
and $J$ has the geometry of Mountain Pass Theorem. This  ensure the existence of a sequence $(u_n) \subset H$ such that
\begin{equation}\label{seq. P-S nivel P.M. func. J}
J(u_n) \longrightarrow c^{**} \ \ \textrm{and} \ \ J'(u_n) \longrightarrow 0
\end{equation}
as $n \rightarrow \infty$, where
$$c^{**} := \inf_{\lambda\in\Lambda}\max_{t\in[0,1]} J(\lambda(t)) > 0$$
and $\Lambda := \{\lambda \in C([0,1],H) : \lambda(0)=0 \ \textrm{and} \ J(\lambda(1))<0\}$.

Evidently, estimates for the minimax level $c^{**}$ analogous to that of Section \ref{estim. minimax} are valid, with hypotheses $(\widehat{f_3})-(\widehat{f_6})$ instead of $(f_3)-(f_6)$, where necessary. Under hypotheses $(f_1)$, $(f_2)$, $(\widehat{f_3})$ and $(f_4^*)$, we also obtain the same conclusions of Proposition \ref{Proposicao seq. de Palais-Smale} for the functional $J$. 

\bigskip

\noindent \textit{Proof of Theorem \ref{Teo. exist. sol. eq. Schr. caso zeta<1}.} Let $(u_n) \subset H$ be the sequence given in \eqref{seq. P-S nivel P.M. func. J}. As in the proof of Theorem \ref{Teorema exist. sol. caso zeta<1}, the boundedness of $(u_n)$ in $H$ implies on the existence of $u_0 \in H$ such that, up to a subsequence, 
\begin{equation}\label{conv. seq. P-S nivel P.M. func. J}
u_n \rightharpoonup u_0 \textrm{ weakly in } H, \ \ u_n \rightarrow u_0 \textrm{ in } L_A^2(\mathbb{R}^2).
\end{equation}
Moreover, as we learned from the proof of Proposition \ref{Proposicao seq. de Palais-Smale}, we have that $Af(u_n) \rightarrow Af(u_0)$ in $L^1(\Omega)$, for any bounded domain $\Omega \subset \mathbb{R}^2$. By this, by the weak convergence in \eqref{conv. seq. P-S nivel P.M. func. J} and the convergence $J'(u_n) \rightarrow 0$, we get
$$J'(u_0)\phi = \left\langle u_0,\phi \right\rangle_H - \int_{\mathbb{R}^2} A(x)f(u_0)\phi = 0, \ \ \forall \ \phi \in C_0^{\infty}(\mathbb{R}^2).$$
By the same arguments of \cite[Theorem 3.22]{AdaFou}, we can verify that $C_0^{\infty}(\mathbb{R}^2)$ is dense in $H$. Hence $J'(u_0)u_0 = 0$. Since, by $(\widehat{f_3})$, we have $J(u_0) \geq (1/\widehat{\theta_0})J'(u_0)u_0$, it follows that $J(u_0) \geq 0$. Hence, we can use the estimate $c^{**} < (2\pi\zeta^2)/\alpha_0$ and proceed as in the proof of Theorem \ref{Teorema exist. sol. caso zeta<1}. \hfill $\Box$ \\

\noindent \textit{Proof of Theorem \ref{Teo. exist. sol. eq. Schr. caso zeta=1}.} It is sufficient to argue as in the the proof of Theorem  Theorem \ref{Teo. exist. sol. eq. Schr. caso zeta<1}, considering now $\zeta=1$ and using the estimate $c^{**} < (2\pi)/\alpha_0.$ \hfill $\Box$

\end{document}